%% file: supports.tex
\documentclass[noinfoline,11pt]{imsart} 

\usepackage{times}
\usepackage{amsthm}
\usepackage[authoryear]{natbib}
\include{header}

\include{macro}

\declaretheorem{theorem}
\declaretheorem{corollary}
\declaretheorem{definition}
\declaretheorem{lemma}

\begin{document}
\title{Aggregation of supports along the Lasso path
}
\runtitle{Aggregation of supports along the Lasso path
}
\author{Pierre C. Bellec}
\runauthor{Bellec}
\address{ENSAE-CREST}

\footnote{
    Accepted for presentation at Conference on
    Learning Theory (COLT) 2016.
}

\begin{abstract}
    In linear regression with fixed design,
    we propose two procedures that aggregate
    a data-driven collection of supports.
    The collection is a subset of the $2^p$ possible supports
    and both its cardinality and its elements can depend on the data.
    The procedures satisfy oracle inequalities with no assumption on the design matrix.
    Then we use these procedures to aggregate the supports that appear on the
    regularization path of the Lasso
    in order to construct an estimator that mimics
    the best Lasso estimator.
    If the restricted eigenvalue condition on the design matrix is satisfied,
    then this estimator achieves optimal prediction bounds.
    Finally, we discuss the computational cost of these procedures.
\end{abstract}

\maketitle

\input{supports-content}

\bibliographystyle{plainnat}
\bibliography{db}

\end{document}

%% file: header.tex
\usepackage[T1]{fontenc}
\usepackage{lmodern}
\usepackage{amssymb,amsmath}
\usepackage{ifxetex,ifluatex}
\usepackage{fixltx2e} 
\IfFileExists{microtype.sty}{\usepackage{microtype}}{}
\IfFileExists{upquote.sty}{\usepackage{upquote}}{}
\ifnum 0\ifxetex 1\fi\ifluatex 1\fi=0 
  \usepackage[utf8]{inputenc}

\usepackage{enumitem}

\usepackage{thmtools}
\usepackage{mathtools}

\usepackage{hyperref}
\usepackage{cleveref}
\usepackage{autonum}

\numberwithin{equation}{section}

\DeclareMathOperator*{\argmin}{argmin}

\newtheorem{problem}{Problem}

%% file: macro.tex
\newcommand{\design}{\mathbb{X}}

\newcommand{\E}{\mathbb{E}}

\newcommand{\jstar}{{k}}
\newcommand{\that}{{\hat \theta}}
\newcommand{\sumj}{\sum_{j=1}^\M}

\newcommand{\RM}{\mathbf{R}^\M}
\newcommand{\Rn}{\mathbf{R}^n}
\newcommand{\R}{\mathbf{R}}
\newcommand{\Rp}{\mathbf{R}^p}
\newcommand{\simplex}{\Lambda^\M}

\newcommand{\K}[1]{\sumj #1_j \log \frac{1}{\pi_{\hat T_j}}}
\newcommand{\pen}[1]{\;\mathrm{pen}_Q( #1 )}

\newcommand{\vmu}{\boldsymbol{\mu}}

\newcommand{\ve}{\boldsymbol{e}}
\newcommand{\vv}{\boldsymbol{v}}
\newcommand{\vu}{\boldsymbol{u}}
\newcommand{\vtheta}{{\boldsymbol{\theta}}}
\newcommand{\vthat}{{\boldsymbol{\hat \theta}}}
\newcommand{\vthetaprime}{{\boldsymbol{\theta'}}}

\newcommand{\hmu}{\boldsymbol{\hat \mu}}
\newcommand{\vf}{\boldsymbol{f}}
\newcommand{\vbeta}{\boldsymbol{\beta}}
\newcommand{\hbeta}{\boldsymbol{\hat \beta}}

\newcommand{\vy}{\mathbf{y}}
\newcommand{\vxi}{{\boldsymbol{\xi}}}
\newcommand{\vdelta}{{\boldsymbol{\delta}}}
\newcommand{\vzero}{{\boldsymbol{0}}}
\newcommand{\supportsHn}{H_{\supports,\hat\sigma^2}}

\newcommand{\euclidnorm}[1]{\vert  #1 \vert_2}
\newcommand{\euclidnorms}[1]{\euclidnorm{ #1 }^2}
\newcommand{\scalednorm}[1]{\Vert  #1 \Vert}
\newcommand{\scalednorms}[1]{\scalednorm{ #1 }^2}

\newcommand{\zeronorm}[1]{| #1 |_0}

\newcommand{\onenorm}[1]{\left\vert  #1 \right\vert_1}

\newcommand{\vertiii}[1]{{\left\vert\kern-0.25ex\left\vert\kern-0.25ex\left\vert #1 
            \right\vert\kern-0.25ex\right\vert\kern-0.25ex\right\vert}}

\newcommand{\opnorm}[1]{\vertiii{ #1 }_{2}}

\newcommand{\hsnorm}[1]{\left\Vert #1 \right\Vert_{\mathrm{F}}}

\newcommand{\Tr}{\mathrm{Tr}}

\newcommand{\lasso}{\hat \vbeta^{\textsc{l}}}
\newcommand{\sqlasso}{\hat \vbeta^{\textsc{sq}}}

\newcommand{\supports}{\hat F}
\newcommand{\M}{{\hat M}}
\newcommand{\T}{{\hat T}}
\newcommand{\das}{{\hmu^{\textsc{q}}_{\supports,\hat\sigma^2}}}
\newcommand{\dassigma}{{\hmu^{\textsc{q}}_{\supports,\sigma^2}}}
\newcommand{\thetaprime}{{\theta'}}
\newcommand{\supp}{{\text{supp}}}

\newcommand{\Crit}{\text{Crit}_{\hat\sigma^2}}
\newcommand{\hmucrit}{\Pi_{\hat T_{\supports,\hat\sigma^2}}(\vy)}

%% file: supports-content.tex
\section{Introduction}
Let $n,p$ be two positive integers.
We consider the mean estimation problem
\begin{equation}
    Y_i = \mu_i + \xi_i, \qquad i=1,...,n,
\end{equation}
where $\vmu =(\mu_1,...,\mu_n)^T \in\Rn$ is unknown,
$\vxi = (\xi_1,...,\xi_n)^T$ is a subgaussian vector, that is,
\begin{equation}
    \E[ \exp(\vv^T\vxi) ] \le \exp \frac{\sigma^2\euclidnorms{\vv}}{2}
    \qquad
    \text{ for all } \vv\in\R^n,
    \label{assum:noise}
\end{equation}
where $\sigma>0$ is the noise level
and $\euclidnorm{\cdot}$ is the Euclidean norm in $\R^n$.
We only observe $\vy = (Y_1,...,Y_n)^T$ and wish to estimate $\vmu$.
A design matrix $\design$ of size $n \times p$ is given and $p$ may be larger than $n$.
We do not require that the model is well-specified, i.e., that
there exists $\vbeta^*\in\Rp$
such that $\vmu = \design \vbeta^*$.
Our goal is to find an estimator $\hmu$
such that the prediction loss $\scalednorms{\hmu - \vmu}$ is small,
where $\scalednorms{\cdot}$ is the empirical loss defined by
\begin{equation}
    \scalednorms{\vu}
    =
    \frac 1 n 
    \euclidnorms{\vu}
    = \frac 1 n \sum_{i=1}^n u_i^2,
    \qquad
    \vu=(u_1,...,u_n)^T\in\R^n.
\end{equation}

In a high-dimensional setting where $p>n$,
the Lasso is known to achieve good prediction performance.
For any tuning parameter $\lambda > 0$,
define the Lasso estimate $\lasso_\lambda$ as any solution of the convex minimization problem
\begin{equation}
    \label{eq:def-lasso}
    \lasso_\lambda \in \argmin_{\vbeta\in\Rp} \frac{1}{2n} \euclidnorms{\vy - \design\vbeta} + \lambda \onenorm{\vbeta},
\end{equation}
where $\onenorm{\vbeta} = \sum_{j=1}^n |\beta_j|$ is the $\ell_1$-norm.
If $\design^T\design/n=I_{p\times p}$ where $I_{p\times p}$ is
the identity matrix of size $p$, then an optimal choice
of the tuning parameter is $\lambda_{univ} \sim \sigma\sqrt{\log(p)/n}$,
up to a numerical constant.
If the Restricted Eigenvalue condition holds (cf. \Cref{def:re} below),
then the universal tuning parameter $\lambda_{univ} \sim \sigma\sqrt{\log(p)/n}$
leads to good prediction performance \citep{bickel2009simultaneous}.
However, 
if the columns of $\design$ are correlated
and the Restricted Eigenvalue condition is not satisfied,
the question
of the optimal choice of the tuning parameter $\lambda$ is still unanswered,
even if the noise level $\sigma^2$ is known.
Empirical and theoretical studies \citep{van2013lasso,hebiri2013correlations,dalalyan2014prediction}
have shown that
if the columns of $\design$ are correlated,
the Lasso estimate with a tuning parameter substantially smaller than
the universal parameter leads to a prediction performance
which is substantially better than that of the Lasso estimate with the universal parameter.
To summarize, these papers raise the following question:

\begin{problem}[Data-driven selection of the tuning parameter]
    \label{problem:1}
    Find a data-driven quantity $\hat \lambda$ such that the prediction loss
    $\scalednorms{\vmu-\design\lasso_{\hat \lambda}}$ is small with high probability.
\end{problem}

In this paper, we focus on a different problem, namely:

\begin{problem}[Lasso Aggregation] 
    \label{problem:2}
    Construct an estimator $\hmu$ that mimics the prediction performance
    of the best Lasso estimator, that is, construct an estimator $\hmu$
    such that with high probability,
    \begin{equation}
        \scalednorms{\hmu - \vmu}
        \le 
        C \min_{\lambda > 0}
        \left(
            \;
            \scalednorms{\design \lasso_\lambda - \vmu }
            +
            \Delta(\lasso_\lambda)
            \;
        \right),
        \label{eq:soi-intro-lasso}
    \end{equation}
    where $C\ge 1$ is a constant and $\Delta(\lasso_\lambda)$ is a small quantity.
\end{problem}

\Cref{problem:1} and \Cref{problem:2} have the same goal,
that is, to achieve a small prediction loss with high probability.
In \Cref{problem:1}, the goal is to select a Lasso estimate that has
small prediction loss. In \Cref{problem:2},
we look for an estimator $\hmu$ such that the prediction performance of $\hmu$
is almost as good as the prediction performance of any Lasso estimate.
The estimator $\hmu$ may be of a different form
than $\lasso_{\hat\lambda}$ for some data-driven parameter $\hat \lambda$.

Our motivation to consider \Cref{problem:2} instead of \Cref{problem:1}
is the following.
Let $\vmu_1,...,\vmu_M$ be deterministic vectors $\R^n$.
If the goal is to mimic the best approximation of $\vmu$ among $\vmu_1,...,\vmu_M$,
it is well known in the literature on aggregation problems 
that an estimator of the form $\hmu = \vf_{\hat k}$ for some data-driven integer $\hat k$
is suboptimal
(cf. Theorem 2.1 in \cite{rigollet2012sparse},
Section 2 of \cite{juditsky2008learning}
and Proposition 6.1 in \cite{gerchinovitz2011prediction}).
Thus, an optimal procedure cannot be valued in the discrete set $\{\vmu_1,...,\vmu_M\}$.
Optimal procedures for this problem are valued in the convex hull of the set $\{\vmu_1,...,\vmu_M\}$.
Examples are the Exponential Weights procedures
proposed in \cite{leung2006information,dalalyan2012sharp}
or the Q-aggregation procedure of \cite{dai2014aggregation}.

Although a lot of progress has been made for various aggregation problems,
to our knowledge no previous work deals with the problem
of aggregation of nonlinear estimators such as
the collection $(\design \lasso_\lambda)_{\lambda > 0}$
based on the sample.
In the setting of the present paper,
the observation $\vy$ and the Lasso estimates are not independent:
no data-split is performed and the same data is used to construct the Lasso estimators
and to aggregate them.

We will show that aggregation of nonlinear estimators of the form $\design \hbeta$
is possible,
for any nonlinear estimators $\hbeta$
and without any assumption on $\design$.
For instance,
an estimator $\hmu$ that achieves \eqref{eq:soi-intro-lasso}
with 
\begin{equation}
    \Delta(\vbeta) \simeq \frac{\sigma^2 \zeronorm{\vbeta}}{n} \log\left( \frac{ep}{\zeronorm{\vbeta}\vee 1}\right)
\end{equation}
is given in \Cref{s:lasso}.
Here, $\zeronorm{\vbeta}$ denotes the number of nonzero coefficients of $\vbeta$ and $a\vee b=\max(a,b)$.

Given a design matrix $\design$,
we call \textit{support} any subset $T$ of $\{1,...,p\}$.
The cardinality of $T$ is denoted by $|T|$ and
for $\vbeta\in\Rp$,
$\supp(\vbeta)$ is the set of indices $k=1,...,p$ such that $\beta_k\ne 0$.
Given a support $T$, we denote by $\Pi_T$ the square matrix of size $n$
which is the orthogonal projection on the linear span of
the columns of $\design$ whose indices belong to $T$.
Denote by $\mathcal P(\{1,...,p\})$ the set of all subsets of $\{1,...,p\}$.
We will consider the following problem.

\begin{problem}[Aggregation of a data-driven collection of supports]
    \label{problem:supports}
    Let $\supports$ be a data-driven collection of supports,
    that is, an estimator valued in $\mathcal{P}(\{1,...,p\})$.
    Construct an estimator $\hmu$ such that with high probability,
    \begin{equation}
        \scalednorms{\hmu - \vmu}
        \le
        \min_{T\in\supports}
        \left(
            \;
            \scalednorms{\Pi_T\vmu  - \vmu}
            + \Delta(T)
            \;
        \right),
        \label{eq:soi-intro-supports}
    \end{equation}
    where $\Delta(\cdot)$ is a function that takes small values.
\end{problem}
The set $\supports$ is a family of supports.
Let us emphasize that both its cardinality and its elements can depend on the
data $\vy$.
Note that for any support $T$,
$\Pi_T \vmu = \design \vbeta^*_T$
where $\vbeta^*_T$ minimizes $\euclidnorms{\design \vbeta - \vmu}$
subject to $\beta_k=0$ for all $k\not\in T$.
In \Cref{s:lasso}, we construct an estimator
$\hmu$ that satisfies 
\eqref{eq:soi-intro-supports} with 
$\Delta(T) \simeq \sigma^2 |T| \log(p/|T|)/n$
for all nonempty supports $T$.
In the literature on aggregation problems, 
one is given a collection of estimators $\{\hmu_1,...,\hmu_M\}$
where $M\ge 1$ is a deterministic integer
and the goal is to mimic the best estimator in this collection,
cf. \cite{tsybakov2014aggregation} and the references therein.
A novelty of the present paper is to consider aggregation
of a collection of estimators, where the cardinality of the collection depends on the data.

The main contributions of the present paper are the following.
\begin{itemize}
    \item In \Cref{s:supports}, we propose an estimator $\das$
        that satisfies the oracle inequality \eqref{eq:soi-intro-supports}
        with
        $\Delta(T) \simeq \hat\sigma^2 |T| \log(p/|T|)/n$ for all nonempty supports $T$,
        where $\hat\sigma^2$ is an estimator of the noise level.
        This estimator solves \Cref{problem:supports}.
        We explain in \Cref{cor:nonlinear} how \Cref{s:supports}
        can be used to construct
        a procedure that aggregates nonlinear estimators of the form
        $\design\hbeta$.
    \item 
        \Cref{s:lasso} is devoted to \Cref{problem:2}.
        Using the result from \Cref{s:supports},
        we construct an estimator $\hmu$
        that satisfies \eqref{eq:soi-intro-lasso} with
        $\Delta(\vbeta) \simeq \sigma^2 \zeronorm{\vbeta} \log(p/\zeronorm{\vbeta})$.
        The computational complexity of the procedure
        is the sum of the complexity of the regularization path of the Lasso
        and the complexity of a convex quadratic program.
\end{itemize}
The proofs can be found in the appendix.

\section{Aggregation of a data-driven family of supports\label{s:supports}}

Throughout this section, let $\supports$ be a data-driven collection of supports
and let $\hat\sigma^2\ge0$ be a real valued estimator.
Let $\M$ be the cardinality of $\supports$, and let $(\hat T_j)_{j=1,...,\M}$
be supports such that
\begin{equation}
    \supports = \{
        \T_1,...,\T_\M 
    \}.
    \label{eq:explicit-supports}
\end{equation}

For all supports $T \subset \{1,...,p\}$, 
define the weights
\cite[]{rigollet2012sparse}
\begin{equation}
    \label{eq:weights}
    \pi_T \coloneqq \left(H_p {p \choose |T|} e^{|T|} \right)^{-1},
    \qquad
    H_p \coloneqq \frac{e-e^{-p}}{e-1}.
\end{equation}
Note that by construction, 
the constant $H_p$ is greater than $1$ and
$\sum_{T\in\mathcal{P}(\{1,...,p\})} \pi_T = 1$ where $\mathcal P(\{1,...,p\})$ is the set of all subsets of $\{1,...,p\}$.
Given a support $T$,
the Least Squares estimator on the linear span of the covariates indexed by $T$
is $\Pi_T \vy$.

We will consider two estimators of $\vmu$ based on $\hat F$ and $\hat\sigma^2$.
The first estimator is defined as follows. Define the criterion
\begin{equation}
    \Crit (T) = \euclidnorms{\vy - \Pi_T\vy} + 18 \hat\sigma^2 \log \frac 1 {\pi_T}.
\end{equation}
We have
\begin{equation}
    \label{eq:weights-log} 
    |T| \le \log \frac 1 {\pi_T} \le \frac 1 2 + 2|T| \log(ep/|T|)
\end{equation}
for any support $T$.
The lower bound is a direct consequence of $H_p>1$
and the upper bound is proved in \cite[(5.4)]{rigollet2012sparse}.
As \eqref{eq:weights-log} holds,
the above criterion is of the same nature as
$C_p$, AIC, BIC and their variants, cf. 
\cite{birge2001gaussian}.
Define the estimator
\begin{equation}
    \label{eq:def-hmucrit}
    \hmucrit
    \qquad
    \text{ where }
    \qquad
    \hat T_{\hat F, \hat\sigma^2} \in \argmin_{ T \in \hat F } \Crit(T).
\end{equation}
The estimator \eqref{eq:def-hmucrit} is the orthogonal projection
of $\vy$ onto the linear span of the columns of $\design$ whose indices are in $\hat T_{\hat F,\hat\sigma^2}$.
If $\hat F$ is not data-dependent, the procedure $\hmucrit$
is close to the one studied in \cite{birge2001gaussian}.

We now define a second estimator valued in the convex hull
of $(\Pi_T \vy)_{T\in \hat F}$.
Let $\M$ be the cardinality of $\supports$, and let  $(\hat T_j)_{j=1,...,\M}$ be supports such that \eqref{eq:explicit-supports} holds.
For any $j=1,...,\M$, let $\hmu_j= \Pi_{\T_j} \vy$.
Define a simplex in $\RM$ as follows:
\begin{equation}
    \simplex = \Big\{
        \vtheta \in \RM, 
        \;\;\; \sumj \theta_j = 1,
        \;\;\; \forall j=1\dots \M, \;\; \theta_j \ge 0
    \Big\}.
\end{equation}
For any $\vtheta\in\RM$, define $\hmu_\vtheta = \sum_{j=1}^\M \theta_j \hmu_j$.
For all $\vtheta\in\simplex$, let
\begin{equation}
    \label{eq:def-H}
    \supportsHn(\vtheta) \coloneqq \euclidnorms{\hmu_\vtheta - \vy} + \frac{1}{2} \pen{\vtheta} + 26 \hat \sigma^2 \K{\theta}.
\end{equation}
where
\begin{equation}
    \label{eq:def-pen-supports}
    \pen{\vtheta} \coloneqq \sum_{j=1}^\M \theta_j \euclidnorms{\hmu_j - \hmu_\vtheta}.
\end{equation}
The penalty \eqref{eq:def-pen-supports} is inspired by recent works on the $Q$-aggregation procedure \cite[]{dai2012deviation}, 
and it was used to derive sharp oracle inequalities
for aggregation of linear estimators \cite[]{dai2014aggregation,bellec2014affine}
and density estimators \cite[]{bellec2014optimal}.
The penalty pushes $\hmu_\vtheta$ 
towards the points $\{\hmu_1,...,\hmu_\M\}$.
Finally, the term $\K{\theta}$ is another penalty that pushes 
the coordinate $\theta_j$ to $0$
if the size of the support $\T_j$ is large.

Define the estimator $\das$ as any minimizer of the function $\supportsHn$ defined
in \eqref{eq:def-H}:
\begin{equation}
    \label{eq:def-das}
    \das \coloneqq \hmu_\vthat, \qquad \vthat \in \argmin_{\vtheta\in\simplex} \supportsHn(\vtheta).
\end{equation}

\begin{theorem}
    \label{thm:supports}
    Let $n,p$ be positive integers and let $\sigma>0$.
    Let $\vmu\in\Rn$ and $\design$ be any matrix of size $n\times p$.
    Let $\supports$ be any data-driven collection of subsets of
    $\{1,...,p\}$.
    Assume that the noise $\vxi$ satisfies \eqref{assum:noise}.
    Let $\hat \sigma^2$ be any real valued estimator
    and let $\delta \coloneqq \mathbb P(\hat\sigma^2 < \sigma^2)$.
    Then for all $x>0$, 
    the estimator $\das$ defined in \eqref{eq:def-das} satisfies
    with probability greater than $1-\delta - 2\exp(-x)$,
    \begin{equation}
        \label{eq:soi-supports}
        \scalednorms{\das - \vmu}
        \le
        \min_{T \in \supports}
        \left(
            \scalednorms{\Pi_{T}\vmu  - \vmu}
            +\frac{\hat\sigma^2}{n}
            \left(
                    24
                + 96 |T| \log\left(
                    \frac{ep}{|T| \vee 1}
                \right)
                \right)
        \right)
        + \frac{22 \sigma^2 x}{n}.
    \end{equation}
    Furthermore, the estimator $\hmucrit$ satisfies
    with probability greater than $1-\delta - 2\exp(-x)$,
    \begin{equation}
        \label{eq:oi-supports}
        \scalednorms{\hmucrit - \vmu}
        \le
        \min_{T \in \supports}
        \left(
            3 \scalednorms{\Pi_{T}\vmu  - \vmu}
            + \frac{\hat\sigma^2}{n}\left(
                    26
                + 104 |T| \log\left(
                    \frac{ep}{|T| \vee 1}
                \right)
                \right)
        \right)
        + \frac{28 \sigma^2 x}{n}.
    \end{equation}
\end{theorem}

In previously studied aggregation problems, 
one is given a collection of estimators $\{\hmu_1,...,\hmu_M\}$
where $M\ge 1$ is a deterministic integer
and the goal is to construct an estimator $\hmu$ such that with high probability,
\begin{equation}
    \scalednorms{\hmu - \vmu}
    \le 
    \min_{j=1,...,M}
    \scalednorms{\hmu_j - \vmu}
    + \Delta_n(M),
\end{equation}
where $\Delta_n(M)$ is a small error term that increases with $M$, 
cf. \cite{tsybakov2014aggregation} and the references therein.
\Cref{thm:supports} is of a different nature for several reasons.
First, the set $\hat F$ is random,
its cardinality can depend on the observed data $\vy$.
Second, the error term that appears
inside the minimum of \eqref{eq:soi-supports} does not depend on the cardinality of $\hat F$.

The estimator $\das$ of \Cref{thm:supports}
with 
$\hat\sigma^2 = \sigma^2$
and $\hat F$ being the set of all subsets of $\{1,...,p\}$
was previously studied as the Exponential Screening estimator \citep{rigollet2011exponential}
or as the Sparsity Pattern Aggregate \citep{rigollet2012sparse}.
In this special case,
$\hat F$
is deterministic and contains all the $2^p$ possible supports.
Because of this exponential number of supports, computing the sparsity pattern
aggregate in practice is hard.
An MCMC algorithm is developed in \cite{rigollet2011exponential}
to compute an approximate solution of the sparsity pattern aggregate,
but to our knowledge there is no theoretical guarantee
that this MCMC algorithm will converge to a good approximation in polynomial time.
The Sparsity Pattern Aggregate 
satisfies \eqref{eq:soi-supports}
with $\hat\sigma^2 = \sigma^2$ and $\hat F =\mathcal P(\{1,...,p\})$.
This sharp oracle inequality yields the minimax rate over all $\ell_q$
balls for all $0<q\le1$,
under no assumption on the design matrix $\design$ \citep{dai2014aggregation,tsybakov2014aggregation}.

To construct the estimator $\das$,
one has to solve the optimization problem \eqref{eq:def-das}.
This is a convex quadratic program of size $|\hat F|$
with a simplex constraint.
The complexity of computing $\das$ is polynomial in the cardinality of $\hat F$.
Thus, if $\hat F$ is small then it is possible to construct $\das$ efficiently.

As the cardinality of $\hat F$ decreases,
the prediction performance of the estimator $\das$
becomes worse, but computing $\das$ becomes easier.

\begin{problem}
    Construct a data-driven set of supports $\hat F$ 
    such that with high probability, there exists a support $T\in\hat F$
    for which, simultaneously,
    the bias $\scalednorms{\Pi_T\vmu -\vmu}$
    and the size $|T|$ are small.
\end{problem}
If we can construct such a set $\hat F$, by \eqref{eq:soi-supports}
the prediction loss of the estimator $\das$ will be small.
Note that \Cref{thm:supports} needs no assumption on the data-driven set $\hat F$
and the design matrix $\design$.

In the following Corollary, we perform aggregation
of a family of nonlinear estimators of the form $(\design\hbeta_k)_{j\in J}$
for some set $J$.
All estimators in the family share the same design matrix $\design$
and this matrix is deterministic.

\begin{corollary}
    \label{cor:nonlinear}
    Let $n,p$ be positive integers and let $\sigma>0$.
    Let $\vmu\in\Rn$ and $\design$ be any matrix of size $n\times p$.
    Let $\supports$ be any data-driven collection of subsets of $\{1,...,p\}$.
    Assume that the noise $\vxi$ satisfies \eqref{assum:noise}.
    Let $(\hbeta_j)_{j\in\hat J}$ be a 
    family of estimators valued in $\Rp$.
    Both the cardinality of the family and its elements can depend on the
    data.
    Let $\hat \sigma^2$ be any real valued estimator
    and let $\delta \coloneqq \mathbb P(\hat\sigma^2 < \sigma^2)$.
    Define $\supports = \{ \supp(\hbeta_j), j\in\hat J\}$
    and let $\das$ be the estimator \eqref{eq:def-das}.
    Then for all $x>0$, 
    the estimator $\das$
    satisfies
    with probability greater than $1-\delta - 2\exp(-x)$,
    \begin{equation}
        \label{eq:soi-nonlinear}
        \scalednorms{\das - \vmu}
        \le
        \min_{j \in \hat J}
        \left(
            \scalednorms{\design\hbeta_j  - \vmu}
            +\frac{\hat\sigma^2}{n}
            \left(
                    24
                    + 96 \zeronorm{\hbeta_j} \log\left(
                    \frac{ep}{\zeronorm{\hbeta_j} \vee 1}
                \right)
                \right)
        \right)
        + \frac{22 \sigma^2 x}{n}.
    \end{equation}
\end{corollary}
Using \eqref{eq:oi-supports},
a similar result can be readily obtained for the estimator $\hmucrit$ with the leading constant $3$.

\section{Aggregation of supports along the Lasso path\label{s:lasso}}

Let us recall some properties of the Lasso path \citep{efron2004least}.
For a given observation $\vy$,
there exists a positive integer  $K$ and a finite sequence
\begin{equation}
    \lambda_0 > \lambda_1 > ... > \lambda_{K} = 0
\end{equation}
such that $\lasso_{\lambda} = \vzero$ for all $\lambda > \lambda_0$,
and such that
\begin{equation}
    \forall \lambda\in(\lambda_{k+1},\lambda_{k}),
    \qquad
    \supp(\lasso_{\lambda}) = \supp(\lasso_{\lambda_k}).
\end{equation}
Thus, there is a finite number of supports on the Lasso path.
In this section, we study the estimator of \Cref{thm:supports}
in the special case $\hat F =\{\supp(\lasso_{\lambda_k}), k=0,...,K\}$,
that is, we aggregate all the supports that appear on the Lasso path.

\begin{theorem}
    \label{thm:path}
    Let $n,p$ be positive integers and let $\sigma>0$.
    Let $\vmu\in\Rn$ and $\design$ be any matrix of size $n\times p$.
    Assume that the noise $\vxi$ satisfies \eqref{assum:noise}.
    Let $\hat \sigma^2$ be any real valued estimator
    and let $\delta \coloneqq \mathbb P(\hat\sigma^2 < \sigma^2)$.
    Let $\lambda_0 >...>\lambda_K$
    be the knots of the Lasso path.
    Let $\supports = \{ \supp(\lasso_{\lambda_j}),  j=0,...,K\}$
    be the family of all supports that appear on the Lasso path
    and let $\das$ be the estimator \eqref{eq:def-das}.
    Then for all $x>0$, 
    the estimator $\das$
    satisfies
    with probability greater than $1-\delta - 2\exp(-x)$,
    \begin{equation}
        \label{eq:soi-path}
        \scalednorms{\das - \vmu}
        \le
        \min_{\lambda>0}
        \left(
            \scalednorms{\design\lasso_\lambda  - \vmu}
            + \frac{\hat \sigma^2}{n}
            \left(
                24
                + 96 \zeronorm{\lasso_\lambda} \log\left(
                    \frac{ep}{\zeronorm{\lasso_\lambda} \vee 1}
                \right)
            \right)
        \right)
        + \frac{22 \sigma^2 x}{n},
    \end{equation}
    where for all $\lambda>0$, $\lasso_\lambda$ is the Lasso estimator \eqref{eq:def-lasso}.
\end{theorem}
Using \eqref{eq:oi-supports},
a similar result can be readily obtained for the estimator $\hmucrit$ with the leading constant $3$.

The
computational complexity of the procedure of \Cref{thm:path}
is polynomial in the number of knots of the Lasso path.
This will be further discussed in \Cref{s:complexity-path}.
In the rest of this section, we assume that $\hat\sigma^2 = \sigma^2$ and $\delta=0$.
We will come back to the estimation of the noise level in \Cref{s:square-root} below.

Interestingly, \Cref{thm:path} does not need any assumption on the design matrix $\design$.
The estimators $\das$ and $\hmucrit$ have a good performance 
as soon as for some possibly unknown $\lambda>0$,
both 
the support of $\lasso_\lambda$
and the loss
$
\scalednorms{\design\lasso_\lambda  - \vmu}
$
are small.

\subsection{Prediction guarantees under the restricted eigenvalue condition}

The goal of this section is to study the prediction 
performance of the procedure defined in \Cref{thm:path} under
the Restricted Eigenvalue condition on the design matrix $\design$.

\begin{definition}
    \label{def:re}
    For any $s\in\{1,...,p\}$ and $c_0>0$, 
    condition $RE(s,c_0)$ is satisfied if 
    \begin{equation}
        \kappa(s,c_0)
        \coloneqq
        \min_{T\subset\{1,...,p\}: |T|\le s}
        \min_{\vdelta\in\R^p: \onenorm{\vdelta_{T^c}} \le c_0 \onenorm{\vdelta_{T}}}
        \frac{\euclidnorm{\design\vdelta}}{\sqrt n\euclidnorm{\vdelta_T}}
        >
        0.
    \end{equation}
\end{definition}

The following result is a reformulation of \citet[Theorem 6.2]{bickel2009simultaneous}.

\begin{theorem}[\citet{bickel2009simultaneous}]
    \label{thm:brt}
    Let $\design$ be such that the diagonal elements of $\design^T\design/n$ are all equal to 1.
    Assume that $\vmu=\design \vbeta^*$ and let $s \coloneqq \zeronorm{\vbeta^*}$.
    Assume that $\vxi\sim\mathcal N(\vzero,\sigma^2 I_{n\times n})$
    and that condition $RE(s,3)$ is satisfied.
    Let $x_0>0$.
    There is an event $\Omega(x_0)$ of 
    probability greater than $1-e^{-x_0}$
    on which the Lasso estimator \eqref{eq:def-lasso}
    with tuning parameter $\lambda_{x_0} = \sigma \sqrt{8(x_0+\log p)/n}$ satisfies
    simultaneously
    \begin{align}
        \zeronorm{\lasso_{\lambda_{x_0}}} &\le \frac{64\phi_{max}}{\kappa^2(s,3)} s,
        \label{eq:brt-ell0}
        \\
        \scalednorms{\design(\lasso_{\lambda_{x_0}}-\vbeta^*)}
        &\le \frac{128\sigma^2 s (x_0+\log p)}{\kappa^2(s,3) n},
        \label{eq:brt-prediction}
    \end{align}
    where $\phi_{max}$ is the largest eigenvalue 
    of the matrix $\design^T\design/n$.
\end{theorem}
Thus, if the restricted eigenvalue condition
is satisfied, the Lasso estimator with the universal parameter
$\lambda_{x_0} = \sigma \sqrt{8(x_0+\log p)/n}$
enjoys simultaneously
an $\ell_0$ norm of the same order as the true sparsity
(cf. \eqref{eq:brt-ell0}),
and a prediction loss of order $s\log(p)/n$
(cf. \eqref{eq:brt-prediction}).

\Cref{cor:path} below is a direct consequence of \Cref{thm:path}
and the bounds \eqref{eq:brt-ell0}-\eqref{eq:brt-prediction}.

\begin{theorem}
    \label{cor:path}
    Let $n,p$ be positive integers and let $\sigma>0$.
    Let $\vmu\in\Rn$ and $\design$ be any matrix of size $n\times p$.
    Let $\supports$ be any data-driven subset of $\{1,...,p\}$.
    Assume that $\vmu=\design \vbeta^*$ and let $s \coloneqq \zeronorm{\vbeta^*}$.
    Assume that $\vxi\sim\mathcal N(\vzero,\sigma^2 I_{n\times n})$
    and that condition $RE(s,3)$ is satisfied.

    Let $\lambda_0 >...>\lambda_K$
    be the knots of the Lasso path.
    Let $\supports = \{ \supp(\lasso_{\lambda_j}),  j=0,...,K\}$
    be the family of all supports that appear on the Lasso path
    and let $\dassigma$ be the estimator \eqref{eq:def-das} with $\hat\sigma^2 = \sigma^2$.
    Then for all $x>0$, 
    the estimator $\dassigma$
    satisfies
    with probability greater than $1-3\exp(-x)$,
    \begin{equation}
        \label{eq:RE-proba}
        \scalednorms{\dassigma- \design\vbeta^*}
        \le
        \frac{(128 + 48\phi_{max}) \sigma^2 s \log p}{\kappa^2(s,3) n}
        +
        \frac{24\sigma^2}{n}
        +
        \frac{128 \sigma^2 s x}{\kappa^2(s,3) n}
        +
        \frac{22\sigma^2 x}{n}
        .
    \end{equation}
    Furthermore, 
    \begin{equation}
        \label{eq:RE-expectation}
        \E 
        \scalednorms{\dassigma- \design\vbeta^*}
        \le
        \frac{(128 + 48\phi_{max}) \sigma^2 s \log p}{\kappa^2(s,3) n}
        +
        \frac{384 \sigma^2 s}{\kappa^2(s,3) n}
        +
        \frac{90\sigma^2}{n}
        .
    \end{equation}
\end{theorem}
Using \eqref{eq:oi-supports},
a similar result can be readily obtained for the estimator $\hmucrit$ with different constants.
\begin{proof}[Proof of \Cref{cor:path}] 
    By \Cref{thm:path} with $\delta=0$,
    there is an event $\Omega_{agg}(x)$ of probability greater than
    $1 - 2 e^{-x}$ such that on $\Omega_{agg}(x)$ we have
    \begin{equation}
        \scalednorms{\dassigma- \design\vbeta^*}
        \le
        \scalednorms{\design(\lasso_{\lambda_x}-\vbeta^*)}
        +
        \frac{\sigma^2}{n}
        \left(
            24
            + 96 \zeronorm{\lasso_{\lambda_x}} \log\left(
                \frac{ep}{\zeronorm{\lasso_\lambda} \vee 1}
            \right)
        \right)
        +
        \frac{22 \sigma^2 x}{n}.
    \end{equation}
    Let $\Omega(x)$ be the event defined in \Cref{thm:brt}.
    Using the simple inequality $\log(p/(\zeronorm{\lasso_\lambda} \vee 1))\le \log p$, and the bounds \eqref{eq:brt-ell0}-\eqref{eq:brt-prediction},
    we obtain that \eqref{eq:RE-proba} holds on the event
    $\Omega_{agg}(x)\cap\Omega(x)$.
    By the union bound, the event
    $\Omega_{agg}(x)\cap\Omega(x)$
    has probability greater than $1-3e^{-x}$.
    Finally,  \eqref{eq:RE-expectation} is obtained 
    from \eqref{eq:RE-proba} by integration.
\end{proof}

The procedure studied in \Cref{cor:path}
aggregates the supports along the Lasso path using the procedure \eqref{eq:def-das}.
A similar result holds for the estimator $\hmucrit$ with a leading constant equal to 3.
\Cref{cor:path} has the following implications.

First, if $x>0$ is fixed,
the prediction performance \eqref{eq:RE-proba} of the estimator $\dassigma$
is similar to that of the Lasso with the universal tuning parameter $\lambda_x$,
up to a multiplicative factor that only involves numerical constants and the quantity $\phi_{max}$.
As soon as $\phi_{max}$
(the operator norm of $\design^T\design/n$) is bounded from above by a constant,
the estimator studied in \Cref{cor:path} enjoys the best known prediction guarantees.

Second, \Cref{cor:path} implies that the estimator $\dassigma$ satisfies
the prediction bound \eqref{eq:RE-proba} simultaneously for all confidence levels.
That is, \eqref{eq:RE-proba} holds for all $x>0$ with probability greater than $1-3e^{-x}$,
in contrast with the Lasso estimator with the universal parameter $\lambda_{x_0}$
which depends on a fixed confidence level $1- e^{-x_0}$.
The Lasso estimator with the universal parameter $\lambda_{x_0}$ satisfies
the prediction bound \eqref{eq:brt-prediction} only for the confidence level $1 - e^{-x_0}$,
but to our knowledge it is not known whether 
the Lasso estimator with the universal parameter $\lambda_{x_0}$
satisfies a similar bound 
for different confidence levels than $1-e^{-x_0}$.
In this regard, the estimator studied in \Cref{cor:path} provides a strict improvement 
compared to the Lasso with the universal parameter.

Third, the estimator $\dassigma$ of \Cref{cor:path} satisfies
the bound \eqref{eq:RE-expectation}, that is, a prediction bound in expectation.
Again, to our knowledge, it is not known whether 
the Lasso estimator with the universal parameter satisfies a similar bound in expectation.

Assuming that the bound \eqref{eq:brt-prediction} is tight
and putting computational issues aside,
the prediction performance of the procedure $\dassigma$ of \Cref{cor:path}
is substantially better than the performance of the Lasso with the universal parameter,
as soon as $\phi_{max}$ is bounded from above by a constant.

An upper bound similar to \eqref{eq:brt-ell0}
is given in \cite[Theorem 3 and Remark 3]{belloni2014pivotal}.
Namely, \cite{belloni2014pivotal} prove that
the square-root Lasso estimator with the universal tuning parameter $\hbeta$
satisfies $\zeronorm{\hbeta} \le C s$ with high probability,
where $s$ is the sparsity of the true parameter
and $C$ is a constant that depends
on the sparse eigenvalues of the matrix $\design^T\design/n$,
cf. \cite[Condition P]{belloni2014pivotal}.
This upper bound can be used instead of \eqref{eq:brt-ell0}
to prove results similar to \eqref{eq:RE-proba}
where $\phi_{max}$ is replaced by a smaller constant that depends
on the sparse eigenvalues of $\design^T\design/n$.

\section{Computational complexity of the Lasso path and $\das$}
\label{s:complexity-path}

Computing the estimator $\das$ of \Cref{thm:path} is done in two steps:
\begin{enumerate}[topsep=0pt,itemsep=-1ex,partopsep=1ex,parsep=1ex]
    \item Compute the full Lasso path and 
        let $\supports = \{\supp(\lambda_0),..., \supp(\lambda_K) \}$
        be all the supports that appear on the Lasso path,
        where $\lambda_0,...,\lambda_K$ are the knots of the Lasso path.
    \item Compute $\das$ as a solution of the quadratic program \eqref{eq:def-das},
        where $\supports$ is defined by Step 1.
\end{enumerate}
(We assume that the complexity of computing $\hat \sigma^2$ 
is negligible compared to the complexity of Step 1 and Step 2 above).
The time complexity of Step 2 is the complexity
of a convex quadratic program of size $|\supports |\le K$,
where $K$ is the number of knots on the Lasso path.
Thus, the global cost of computing the estimator $\das$ of \Cref{thm:path}
is polynomial in $K$.

There exist efficient algorithms to compute the entire Lasso path \citep{efron2004least}.
However, 
\cite{mairal2012complexity}
proved that for some values of $\design$ and $\vy$,
the regularization path
of the Lasso
contains more than $3^p/2$ knots.
Hence, for some design matrix $\design$ and some observation $\vy$,
an exact computation of the full Lasso path is not realizable in polynomial time.
In order to fix this computational issue,
\cite{mairal2012complexity} propose an
algorithm that computes an approximate regularization path for the Lasso.
For some fixed $\epsilon>0$,
this algorithm is guaranteed to terminate with less than $O(1/\sqrt{\epsilon})$ knots
and the points on the approximate path have a duality gap smaller than $\epsilon$.
This approximation algorithm can
be used instead of computing the exact Lasso path.
That is, one may compute the estimator $\das$ where $\supports$ is the
collection of supports that appear on the approximate path
computed by the algorithm of \cite{mairal2012complexity}.

Another solution to avoid computational issues is as follows.
Let $M$ be a positive integer.
Instead of computing the Lasso path, one may consider
a grid of tuning parameters $\lambda_1,...,\lambda_M >0$
and aggregate the supports of corresponding Lasso estimates
$\lasso_{\lambda_1},...\lasso_{\lambda_M}$.
The advantage of this approach is twofold.
First, for all $j=1,...,M$
the Lasso estimate $\lasso_{\lambda_j}$
can be computed by standard convex optimization solvers.
Second, the time complexity of the procedure is guaranteed to be polynomial in $M$ and $p$.
For any $x>0$,
by \Cref{cor:nonlinear}, this procedure satisfies,
with probability greater than $1-3e^{-x}$
\begin{equation}
    \scalednorms{\das - \vmu}
    \le
    \min_{j=1,...,M}
    \left(
        \scalednorms{\design\lasso_{\lambda_j}  - \vmu}
        +
        \frac{\hat\sigma^2}{n}
        \left(
            24
            + 96 \zeronorm{\lasso_{\lambda_j}} \log\left(
                \frac{ep}{\zeronorm{\lasso_{\lambda_j}} \vee 1}
            \right)
        \right)
    \right)
    + \frac{22 \sigma^2 x}{n}.
\end{equation}
This oracle inequality is not a strong as \eqref{eq:soi-path}.
However, if at least one of the Lasso estimates $\{\lasso_{\lambda_j}, j=1,...,M\}$
enjoys a small prediction loss and a small $\ell_0$ norm,
then the prediction loss of $\das$ is also small.

\section{A fully data-driven procedure using the Square-Root Lasso}
\label{s:square-root}

This section proposes a fully data-driven procedure, based on the Square-Root Lasso.
The choice of grid comes from the empirical and theoretical observations
that for a correlated design matrix,
there exists a tuning parameter smaller than
the universal parameter which enjoys better prediction performance
than the universal parameter
\citep{van2013lasso,hebiri2013correlations,dalalyan2014prediction}.

\begin{enumerate}[topsep=0pt,itemsep=-1ex,partopsep=1ex,parsep=1ex]
    \item 
        Let $\lambda_{\max} = 2 \sqrt{\log(p/0.01)/n}$ be the universal parameter of the Square-Root Lasso \citep{belloni2014pivotal}
        with confidence level $0.01$.
    \item
        Let $\lambda_{\min}$ be a conservatively small value of the tuning parameter.
    \item Let $M$ be an integer.
    \item Consider the geometric grid $\{\lambda_1,...,\lambda_M\}$ such that 
        \begin{equation}
            \lambda_j = \lambda_{\min} \left(\frac{\lambda_{\max}}{\lambda_{\min}}\right)^{((j-1)/M-1)},
            \qquad
            j=1,...,M.
        \end{equation}
    \item Compute the Square-Root Lasso estimators $\sqlasso_{\lambda_1},...\sqlasso_{\lambda_M}$
        with parameters $\lambda_1,...,\lambda_M$
        (it is possible to perform this computation simultaneously for all $\lambda_1,...,\lambda_M$, cf. \cite{pham2014robust} and the references therein).
    \item Let $\hat F = \{\supp(\sqlasso_{\lambda_j}), j=1,...,M\}$ be the supports of the computed Square-Root Lasso estimators.
    \item Let $\hat \sigma^2$ be the variance estimated by the Square-Root Lasso with the universal parameter $\lambda_{\max}$.
    \item For this choice of $\hat \sigma^2$ and $\hat F$, return the estimator $\das$ or the estimator $\hmucrit$ .
\end{enumerate}
This estimator $\das$ returned by this procedure enjoys the theoretical guarantee
\begin{equation}
    \scalednorms{\das - \vmu}
    \le
    \min_{j=1,...,M}
    \left(
        \scalednorms{\design\sqlasso_{\lambda_j}  - \vmu}
        +\frac{\hat\sigma^2}{n}\left(
            24
            + 96 \zeronorm{\sqlasso_{\lambda_j}} \log\left(
                \frac{ep}{\zeronorm{\sqlasso_{\lambda_j}} \vee 1}
            \right)
        \right)
    \right)
    + \frac{22 \sigma^2 x}{n}
\end{equation}
with probability greater than $1-3e^{-x}$.
A similar guarantee with leading constant 3
can be obtained for the estimator $\hmucrit$ using
\eqref{eq:oi-supports}.

\section{Concluding remarks}

We have presented two procedures \eqref{eq:def-hmucrit} and \eqref{eq:def-das}
that aggregates a data-driven collection of supports $\hat F$.
These procedures satisfy the oracle inequalities given in \Cref{thm:supports} above,
which is the main result of the paper.
\Cref{s:lasso,s:complexity-path} study the situation
where $\hat F$ is the collection of supports that appear along the Lasso path.
These procedures may be used for other data-driven collections $\hat F$ as well.

These procedures allow one to perform a trade-off between
prediction performance and computational cost.
If $\hat F$ contains all the $2^p$ supports, 
these procedures achieve optimal prediction guarantees with
no assumption on the design matrix $\design$, but can not be
realized in polynomial time.
On the other hand, if the cardinality of $\hat F$ is small
(say, polynomial in $n$ and $p$),
then it is possible to compute the estimators \eqref{eq:def-hmucrit} and \eqref{eq:def-das}
in polynomial time.
In view of \eqref{eq:soi-intro-lasso},
one should look for a data-driven set $\hat F$ with the following properties.
\begin{enumerate}[topsep=5pt,itemsep=-0.5ex,partopsep=1ex,parsep=1ex]
    \item 
        The set $\hat F$ is small so that the estimators
        \eqref{eq:def-das} and \eqref{eq:def-hmucrit}
        can be computed rapidly,
    \item The set 
        $\hat F$ contains a support $T$
        such that $|T|$ and $\scalednorms{\pi_T\vmu -\vmu}$ are simultaneously small, 
        so that
        the procedures \eqref{eq:def-das} and \eqref{eq:def-hmucrit}
        enjoy good prediction performance.
\end{enumerate}
A natural choice for $\hat F$ is the collection of supports that appear along the Lasso
path.
This choice of $\hat F$ was studied in \Cref{s:lasso,s:complexity-path}.
Another natural choice is to aggregate the supports of several hard-thresholded Lasso estimators,
since the hard-thresholded Lasso is sign-consistent
under weak conditions on the design \cite[Definition 5 and Corollary 2]{meinshausen2009lasso}.
Further research will investigate other means to construct a data-driven collection $\hat F$
such that the above two properties are satisfied.


\subsection*{Acknoledgements}
We would like to thank Alexandre Tsybakov for helpful comments during the writing of this manuscript.

\appendix

\section{Proof of \Cref{thm:supports}}

For any matrix $A\in\R^{n\times n}$, define the operator norm of $A$ and the Frobenius
norm of $A$ by
\begin{equation}
\opnorm{A} \coloneqq \sup_{\euclidnorms{\vu}=1} \euclidnorm{A\vu},
\qquad
\hsnorm{A} = \sqrt{\Tr(A^TA)},
\end{equation}
respectively.

\begin{proof}[Proof of \eqref{eq:soi-supports}] 
    For all $S,T\subset\{1,...,p\}$, define the event
    \begin{equation}
        \Omega_{S,T}
        = \left\{
                Z(S,T)
                \le 
                4 \sigma^2 |S|
                +
                22\sigma^2 \left(\log \frac 1 {\pi_S\pi_S} + x\right)
            \right\},
    \end{equation}
    where 
    \begin{equation}
        Z(S,T) = 
        2\vxi^T(\Pi_S\vy - \Pi_T\vmu)
        - \frac 1 2 \euclidnorms{\Pi_S\vy -\Pi_T\vy}
        .
        \label{eq:def-Z}
    \end{equation}
    Define the event $\mathcal V \coloneqq \{ \hat\sigma^2 \ge \sigma^2 \}$.
    On the event $\mathcal A \coloneqq \mathcal V \cap ( \cap_{S,T\subset\{1,...,p\}} \Omega_{S,T})$,
    we have simultaneously for all supports $S,T$
    \begin{equation}
            Z(S,T) - 26 \hat\sigma^2 \log \frac 1 {\pi_S} - 22 \sigma^2  \log \frac 1 {\pi_T}
        \le 
            22\sigma^2 x + 4\sigma^2 |S| - 4 \sigma^2 \log \frac 1 {\pi_S}
        \le 22 \sigma^2 x
    \end{equation}
    where we have used 
    that 
    $\log \frac 1 {\pi_S}
    \ge |S|
    $, cf. \eqref{eq:weights-log}.
    By Lemma \ref{lemma:technical-as},
    on the event $\mathcal A$ we have
    \begin{equation}
        \euclidnorms{\das - \vmu}
        \le
        \min_{T \in \supports}
        \left(
            \euclidnorms{\Pi_{T}\vmu  - \vmu}
            + (26\hat\sigma^2 + 22\sigma^2) \log \frac 1 {\pi_T}
        \right)
        + 22 \sigma^2 x.
    \end{equation}
    To obtain \eqref{eq:soi-supports}, we use \eqref{eq:weights-log} and
    the fact that on the event $\mathcal V$,
    $26\hat\sigma^2 + 22\sigma^2\le 48\hat\sigma^2$.

    It remains to bound from below the probability of the event  $\mathcal A$.
    Denote by $\mathcal B^c$ the complement of any event $\mathcal B$.
    We proceed with the union bound as follows,
    \begin{equation}
        \mathbb P (\mathcal A^c)
        \le 
        \mathbb P (\mathcal V^c)
        + \sum_{S,T\subset\{1,...,p\}}
        \mathbb P (\Omega_{S,T}^c).
    \end{equation}
    By definition, $\delta = 
        \mathbb P (\mathcal V^c)$
    and for any $S,T\subset\{1,...,p\}$,
    Lemma \ref{lemma:technical-sigmax} with $t = x + \log \frac 1 {\pi_S\pi_T}$
    yields that  $\mathbb P (\Omega_{S,T}^c) \le \pi_S\pi_T 2 \exp(-x)$.
    As $
    \sum_{S,T\subset\{1,...,p\}}\pi_S \pi_T 
    = (\sum_{S\subset\{1,...,p\}}\pi_S)^2 
    = 1$, 
    we have established that
    \begin{equation}
        \mathbb P (\mathcal A^c)
        \le 
        \delta + 2 \exp(-x).
    \end{equation}
\end{proof}

The proof of \eqref{eq:oi-supports} is close to the argument used in \cite{birge2001gaussian},
cf. \cite[Section 2.3]{giraud2014introduction} for a recent reference on model selection.
The novelty of the present paper is to consider a data-driven collection of estimators.

\begin{proof}[Proof of \eqref{eq:oi-supports}] 
    Let $\hat\Lambda= 18 \hat\sigma^2$
    and let $\hat T = \hat T_{\hat F, \hat\sigma^2}$ for notational simplicity.
    By definition of $\hmucrit=\Pi_{\hat T}\vy$, for all $T\in\hat F$ we have
    $\Crit(\hat T)\le \Crit(T)$ which can be rewritten as
    \begin{align}
        \euclidnorms{\hmucrit - \vmu}
        + \hat \Lambda \log \frac 1 {\pi_{\hat T}}
        &\le
        \euclidnorms{\Pi_T\vy - \vmu}
        + \hat \Lambda \log \frac 1 {\pi_{T}}
        + 2 \vxi^T(\Pi_{\hat T}\vy- \Pi_T\vy),
        \\
        &\le
        \euclidnorms{\Pi_T\vmu - \vmu}
        + \hat \Lambda \log \frac 1 {\pi_{T}}
        + 2 \vxi^T\Pi_{\hat T} \vxi
        +2 \vxi^T (
            \Pi_{\hat T}\vmu 
            - \Pi_T\vmu
        ) 
         - \euclidnorms{\Pi_T\vxi}.
        \label{eq:fjeiwoaa}
    \end{align}
    Define the event $\mathcal V \coloneqq \{ \hat\sigma^2 \ge \sigma^2 \}$.
    For all $S,T\subset\{1,...,p\}$, define
    \begin{align}
    W(S) & = 2\vxi^T \Pi_{S}\vxi - 10 \sigma^2 \log \frac 1 {\pi_{S}},
        \\
    W'(S,T) & = 2\vxi^T(\Pi_S\vmu - \Pi_T\vmu) - 8 \sigma^2 \log \frac 1 {\pi_S\pi_T} - \frac 1 4 \euclidnorms{\Pi_S\vmu - \Pi_T\vmu}.
    \end{align}
    With this notation, using the simple inequality $
     - \euclidnorms{\Pi_T\vxi} \le 0$,
    \eqref{eq:fjeiwoaa} implies that on the event $\mathcal V$,
    \begin{align}
        \euclidnorms{\hmucrit - \vmu}
        &\le
        \euclidnorms{\Pi_T\vmu - \vmu}
        + \hat \Lambda \log \frac 1 {\pi_{T}}
        + 8 \sigma^2 \log \frac 1 {\pi_{T}}
        + W(\hat T) + W'(\hat T,T)
        + \frac 1 4 \euclidnorms{\Pi_{\hat T}\vmu - \Pi_{T}\vmu},
    \end{align}
    Using that $\euclidnorms{\Pi_{\hat T}\vmu - \Pi_{T}\vmu} \le 2\euclidnorms{\Pi_{\hat T}\vmu - \vmu} + 2 \euclidnorms{\vmu  - \Pi_{T}\vmu}$
    and that $\euclidnorms{\Pi_{\hat T}\vmu - \vmu} \le \euclidnorms{\Pi_{\hat T}\vy - \vmu}$,
    we obtain
    \begin{align}
        \frac 1 2 \euclidnorms{\hmucrit - \vmu}
        &\le
        \frac 3 2 \euclidnorms{\Pi_T\vmu - \vmu}
        + \hat \Lambda \log \frac 1 {\pi_{T}}
        + 8 \sigma^2 \log \frac 1 {\pi_{T}}
        + W(\hat T) + W'(\hat T,T).
    \end{align}
    For all $S,T\subset \{1,...,p\}$, define the events
    \begin{equation}
        \Omega_S \coloneqq \{ W(S) \le 6 \sigma^2 x \},
        \qquad
        \Omega_{S,T} \coloneqq \{ W'(S,T) \le 8 \sigma^2 x\}.
    \end{equation}
    On the event $\mathcal V\cap (\cap_{S\subset\{1,...,p\}} \Omega_S) \cap (\cap_{S,T\subset\{1,...,p\}} \Omega_{S,T})$,
    \eqref{eq:oi-supports} holds.
    It remains to bound from below the probability of this event.

    For any fixed $S\subset\{1,...,p\}$,
    using \eqref{eq:weights-log} and \eqref{eq:hsu-consequence} with $t = x+ \log \frac 1 {\pi_S}$ 
    we have $\mathbb P(\Omega_S^c) \le \pi_S e^{-x}$.

    Let $S,T\subset\{1,...,p\}$ be fixed.
By using \eqref{eq:chernoff} with $\vv = 2( \Pi_S\vmu - \Pi_T\vmu))$ and $t=x + \log \frac 1 {\pi_S\pi_T}$,
    we have that on an event of probability greater than $1-\pi_S\pi_Te^{-x}$,
    \begin{equation}
        2\vxi^T(\Pi_S\vmu - \Pi_T\vmu)  \le
        2 \sigma \sqrt{2(x+\log(1/ {\pi_S\pi_T}))} \euclidnorm{\Pi_S\vmu - \Pi_T\vmu}
        \le 
        8\sigma^2 \left(x+\log\frac 1 {\pi_S\pi_T}\right)
        + \frac 1 4 \euclidnorms{\Pi_S\vmu - \Pi_T\vmu}.
    \end{equation}
    Thus, $\mathbb P(\Omega_{S,T}^c) \le \pi_S\pi_Te^{-x}$.

    As in the proof of \eqref{eq:soi-supports}, the union bound completes the proof.
\end{proof}

\section{Technical Lemmas}

\begin{lemma}
    \label{lemma:technical-as}
    For any estimator $\hat\sigma^2$,
    let $\vthat$ be a minimizer of
    \eqref{eq:def-H}.
    Then, almost surely,
    \begin{align}
        \euclidnorms{\hmu_\vthat -\vmu}
        \le 
        \min_{\jstar=1,...,\M}
        \left(
            \euclidnorms{\Pi_{\T_\jstar}\vmu -\vmu}
            + (26 \hat\sigma^2 + 22\sigma^2) \log\frac{1}{\pi_{\T_\jstar}}
        \right)
        + W,
        \label{eq:technical-soi}
    \end{align}
    where
    \begin{equation}
        W \coloneqq
        \max_{S,T\in\supports} 
        \left(
            Z(S, T)
            - 26 \hat\sigma^2 \log \frac 1 {\pi_S}
            - 22 \sigma^2 \log \frac 1 {\pi_T}
        \right)
    \end{equation}
    and $Z(\cdot,\cdot)$ is defined in \eqref{eq:def-Z}.
\end{lemma}

\begin{proof}[Proof of Lemma \ref{lemma:technical-as}]
    Let $\hat \Lambda = 26 \hat\sigma^2$.
    The function $\supportsHn$ is convex and differentiable,
    it can be rewritten as
    \begin{equation}
        \forall \vtheta\in\simplex,\; \supportsHn(\vtheta)  = \frac{1}{2} \euclidnorms{\hmu_\vtheta} 
        + \euclidnorms{\vy}
        + \sumj \theta_j 
        \left(
             - 2 \vy^T \hmu_j
             + \frac{1}{2} \euclidnorms{\hmu_j} + \hat\Lambda \log\frac{1}{\pi_{\T_j}}
        \right).
    \end{equation}
    By simple algebra, for any $\vthetaprime\in\RM$,
    \begin{align}
        &\nabla \supportsHn(\vthat)^T \vthetaprime  = \hmu_\vthat^T \hmu_\vthetaprime + 
        \sumj \thetaprime_j
        \left(
             - 2 \vy^T \hmu_j  
             + \frac{1}{2} \euclidnorms{\hmu_j} + \hat\Lambda \log\frac{1}{\pi_{\T_j}}
         \right)
         ,  \label{eq:proof-thetaprime} \\
         &\nabla \supportsHn(\vthat)^T (- \vthat)   = - \euclidnorms{\hmu_\vthat - \vmu} + \euclidnorms{\vmu} +
         \sumj \that_j
        \left(
             2 \vxi^T \hmu_j 
             - \frac{1}{2} \euclidnorms{\hmu_j} - \hat\Lambda \log\frac{1}{\pi_{\T_j}}
         \right),
    \end{align}
    By summing the last 
    display
    and 
    equality \eqref{eq:proof-thetaprime} applied to $\vthetaprime = \ve_\jstar$, we get
    \begin{multline}
        \nabla \supportsHn(\vthat)^T (\ve_\jstar - \vthat)
        = - \euclidnorms{\hmu_\vthat - \vmu} + \euclidnorms{\hmu_\jstar - \vmu}
        + \hat\Lambda \log \frac 1 {\pi_{\hat T_\jstar}}
        \\
        + \sum_{j=1}^\M
        \that_j
        \Bigg[
            2 \vxi^T (\hmu_j-\hmu_\jstar) 
            - \frac{1}{2} \euclidnorms{\hmu_j - \hmu_\jstar}
            - \hat\Lambda\log\frac 1 {\pi_{\T_j}}
            \Bigg].
    \end{multline}
    Since $\hmu_\jstar=\Pi_{\hat T_\jstar}\vy$ is a Least Squares estimator
    over the linear span of the covariates in $\hat T_\jstar$, we have
    $\euclidnorms{\hmu_\jstar - \vy}\le \euclidnorms{\Pi_{\hat T_\jstar}\vmu -\vy}$ which can be rewritten as
    \begin{equation}
        \euclidnorms{\hmu_\jstar - \vmu}
        \le
        \euclidnorms{\Pi_{\hat T_\jstar}\vmu - \vmu}
        + 2 \vxi^T(\hmu_\jstar - \Pi_{\hat T_\jstar}\vmu).
    \end{equation}
    We thus have
    \begin{multline}
        \nabla \supportsHn(\vthat)^T (\ve_\jstar - \vthat)
        \le - \euclidnorms{\hmu_\vthat - \vmu} + \euclidnorms{\Pi_{\hat T_\jstar}\vmu - \vmu}
        + (\hat\Lambda + 22\sigma^2) \log \frac 1 {\pi_{\hat T_\jstar}}
        \\
        + \sum_{j=1}^\M
        \that_j
        \Bigg[
            2 \vxi^T (\hmu_j-\Pi_{\hat T_\jstar}\vmu) 
            - \frac{1}{2} \euclidnorms{\hmu_j - \hmu_\jstar}
            - \hat\Lambda\log\frac 1 {\pi_{\T_j}}
            - 22\sigma^2 \log\frac 1 {\pi_{\T_\jstar}}
            \Bigg].
    \end{multline}
    For all $\jstar = 1,...,\M$,
    \cite[Section 4.2.3]{boyd2009convex} yields
    $\nabla \supportsHn(\vthat)^T (\ve_\jstar - \vthat) \ge 0$.
    Furthermore, a linear function over the simplex is maximized at a vertex,
    so 
    almost surely we obtain
    \eqref{eq:technical-soi}.
\end{proof}

\begin{lemma}
    \label{lemma:technical-sigmax}
    Let $t>0$.
    For any supports $S,T\subset \{1,...,p\}$,
    the quantity $Z(S,T)$ defined in \eqref{eq:def-Z}
    satisfies
    with probability greater than $1-2\exp(-t)$,
    \begin{equation}
        Z(S,T) 
        \le 4 \sigma^2 |S| +
        22 \sigma^2 t.
    \end{equation}
\end{lemma}
\begin{proof}[Proof of Lemma \ref{lemma:technical-sigmax}] 
    Let $D=\Pi_S-\Pi_T$. Then almost surely,
    \begin{equation}
        Z(S,T)
        = 2 \vxi^T \Pi_S\vxi
        + \vxi^T (2 D\vmu - D^2\vmu)
        - \frac 1 2 \euclidnorms{D\vmu}
        - \frac 1 2 \euclidnorms{D\vxi}.
    \end{equation}
    It is clear that  $- \euclidnorms{D\vxi}\le 0$.
    As $\vxi$ satisfies \eqref{assum:noise},
    a Chernoff bound yields that for all $\vv\in\R^n$,
    \begin{equation}
        \label{eq:chernoff}
        \mathbb P
        \left(\vxi^T \vv > \sigma \euclidnorm{\vv} \sqrt{2t}\right) \le \exp(-t).
    \end{equation}
    It is clear that $\opnorm{D}\le 2$.
    We apply this concentration inequality to $\vv = 2D\vmu-D^2\vmu$ to get
    that  with probability greater than $1-\exp(-t)$,
    \begin{align}
        \vxi^T (2 D\vmu - D^2\vmu)
        \le
        \sigma \euclidnorm{2D\vmu - D^2\vmu}
        \sqrt{2t} 
        & \le
        \sigma \opnorm{2I_n - D} \euclidnorm{D\vmu}
        \sqrt{2t}, \\
        &\le
        \sigma 4 \euclidnorm{D\vmu}
        \sqrt{2t}
        \le 
        16 \sigma^2 t
        + \frac 1 2 \euclidnorms{D\vmu}
        .
    \end{align}
    Finally,
    let $r\le |S|$ be the rank of $\Pi_S$.
    The matrix $\Pi_S$ is an orthogonal projector.
    Hence
    $\hsnorm{\Pi_S}^2 = r$ and $\opnorm{\Pi_S} \le 1$,
    so that applying the concentration inequality from 
    \cite{hsu2012tail} yields
    that with probability greater than $1-\exp(-t)$,
    \begin{equation}
        \label{eq:hsu-consequence}
        2 \vxi^T \Pi_S\vxi
        \le 2\sigma^2(
            r
            + 2 \sqrt{r t}
            + 2 t
        )
        \le 4 \sigma^2 r + 6 \sigma^2 t
        \le 4 \sigma^2 |S| + 6 \sigma^2 t.
    \end{equation}
    A union bound completes the proof.
\end{proof}

%% file: supports.bbl
\begin{thebibliography}{24}
\providecommand{\natexlab}[1]{#1}
\providecommand{\url}[1]{\texttt{#1}}
\expandafter\ifx\csname urlstyle\endcsname\relax
  \providecommand{\doi}[1]{doi: #1}\else
  \providecommand{\doi}{doi: \begingroup \urlstyle{rm}\Url}\fi

\bibitem[Bellec(2014{\natexlab{a}})]{bellec2014affine}
Pierre~C. Bellec.
\newblock Optimal bounds for aggregation of affine estimators.
\newblock \emph{arXiv:1410.0346, Submitted}, 2014{\natexlab{a}}.
\newblock URL \url{http://arxiv.org/abs/1410.0346}.

\bibitem[Bellec(2014{\natexlab{b}})]{bellec2014optimal}
Pierre~C. Bellec.
\newblock Optimal exponential bounds for aggregation of density estimators.
\newblock \emph{Accepted in Bernoulli}, 2014{\natexlab{b}}.
\newblock URL \url{http://arxiv.org/abs/1405.3907}.

\bibitem[Belloni et~al.(2014)Belloni, Chernozhukov, and
  Wang]{belloni2014pivotal}
Alexandre Belloni, Victor Chernozhukov, and Lie Wang.
\newblock Pivotal estimation via square-root {L}asso in nonparametric
  regression.
\newblock \emph{Ann. Statist.}, 42\penalty0 (2):\penalty0 757--788, 2014.
\newblock ISSN 0090-5364.
\newblock \doi{10.1214/14-AOS1204}.
\newblock URL \url{http://dx.doi.org/10.1214/14-AOS1204}.

\bibitem[Bickel et~al.(2009)Bickel, Ritov, and
  Tsybakov]{bickel2009simultaneous}
Peter~J. Bickel, Ya'acov Ritov, and Alexandre~B. Tsybakov.
\newblock Simultaneous analysis of lasso and {D}antzig selector.
\newblock \emph{Ann. Statist.}, 37\penalty0 (4):\penalty0 1705--1732, 2009.
\newblock ISSN 0090-5364.
\newblock \doi{10.1214/08-AOS620}.
\newblock URL \url{http://dx.doi.org/10.1214/08-AOS620}.

\bibitem[Birg{\'e} and Massart(2001)]{birge2001gaussian}
Lucien Birg{\'e} and Pascal Massart.
\newblock Gaussian model selection.
\newblock \emph{J. Eur. Math. Soc. (JEMS)}, 3\penalty0 (3):\penalty0 203--268,
  2001.
\newblock ISSN 1435-9855.
\newblock \doi{10.1007/s100970100031}.
\newblock URL \url{http://dx.doi.org/10.1007/s100970100031}.

\bibitem[Boyd and Vandenberghe(2009)]{boyd2009convex}
Stephen Boyd and Lieven Vandenberghe.
\newblock \emph{Convex optimization}.
\newblock Cambridge university press, 2009.

\bibitem[Dai et~al.(2012)Dai, Rigollet, and Zhang]{dai2012deviation}
Dong Dai, Philippe Rigollet, and Tong Zhang.
\newblock Deviation optimal learning using greedy {$Q$}-aggregation.
\newblock \emph{Ann. Statist.}, 40\penalty0 (3):\penalty0 1878--1905, 2012.
\newblock ISSN 0090-5364.
\newblock \doi{10.1214/12-AOS1025}.
\newblock URL \url{http://dx.doi.org/10.1214/12-AOS1025}.

\bibitem[Dai et~al.(2014)Dai, Rigollet, Xia, and Zhang]{dai2014aggregation}
Dong Dai, Philippe Rigollet, Lucy Xia, and Tong Zhang.
\newblock Aggregation of affine estimators.
\newblock \emph{Electron. J. Statist.}, 8\penalty0 (1):\penalty0 302--327,
  2014.
\newblock \doi{10.1214/14-ejs886}.
\newblock URL \url{http://dx.doi.org/10.1214/14-ejs886}.

\bibitem[Dalalyan and Salmon(2012)]{dalalyan2012sharp}
Arnak~S. Dalalyan and Joseph Salmon.
\newblock Sharp oracle inequalities for aggregation of affine estimators.
\newblock \emph{Ann. Statist.}, 40\penalty0 (4):\penalty0 2327--2355, 2012.
\newblock ISSN 0090-5364.
\newblock \doi{10.1214/12-AOS1038}.
\newblock URL \url{http://dx.doi.org/10.1214/12-AOS1038}.

\bibitem[Dalalyan et~al.(2014)Dalalyan, Hebiri, and
  Lederer]{dalalyan2014prediction}
Arnak~S Dalalyan, Mohamed Hebiri, and Johannes Lederer.
\newblock On the prediction performance of the lasso.
\newblock \emph{arXiv preprint arXiv:1402.1700}, 2014.

\bibitem[Efron et~al.(2004)Efron, Hastie, Johnstone, and
  Tibshirani]{efron2004least}
Bradley Efron, Trevor Hastie, Iain Johnstone, and Robert Tibshirani.
\newblock Least angle regression.
\newblock \emph{Ann. Statist.}, 32\penalty0 (2):\penalty0 407--499, 2004.
\newblock ISSN 0090-5364.
\newblock \doi{10.1214/009053604000000067}.
\newblock URL \url{http://dx.doi.org/10.1214/009053604000000067}.
\newblock With discussion, and a rejoinder by the authors.

\bibitem[Gerchinovitz(2011)]{gerchinovitz2011prediction}
S{\'e}bastien Gerchinovitz.
\newblock \emph{Prediction of individual sequences and prediction in the
  statistical framework: some links around sparse regression and aggregation
  techniques}.
\newblock PhD thesis, Universit{\'e} Paris Sud-Paris XI, 2011.
\newblock URL \url{https://tel.archives-ouvertes.fr/tel-00653550}.

\bibitem[Giraud(2015)]{giraud2014introduction}
Christophe Giraud.
\newblock \emph{Introduction to high-dimensional statistics}, volume 139 of
  \emph{Monographs on Statistics and Applied Probability}.
\newblock CRC Press, Boca Raton, FL, 2015.
\newblock ISBN 978-1-4822-3794-8.

\bibitem[Hebiri and Lederer(2013)]{hebiri2013correlations}
Mohamed Hebiri and Johannes Lederer.
\newblock How correlations influence lasso prediction.
\newblock \emph{IEEE Trans. Inform. Theory}, 59\penalty0 (3):\penalty0
  1846--1854, March 2013.
\newblock \doi{10.1109/tit.2012.2227680}.
\newblock URL \url{http://dx.doi.org/10.1109/tit.2012.2227680}.

\bibitem[Hsu et~al.(2012)Hsu, Kakade, and Zhang]{hsu2012tail}
Daniel Hsu, Sham~M. Kakade, and Tong Zhang.
\newblock A tail inequality for quadratic forms of subgaussian random vectors.
\newblock \emph{Electron. Commun. Probab.}, 17:\penalty0 no. 52, 6, 2012.
\newblock ISSN 1083-589X.
\newblock \doi{10.1214/ECP.v17-2079}.
\newblock URL \url{http://dx.doi.org/10.1214/ECP.v17-2079}.

\bibitem[Juditsky et~al.(2008)Juditsky, Rigollet, and
  Tsybakov]{juditsky2008learning}
A.~Juditsky, P.~Rigollet, and A.~B. Tsybakov.
\newblock Learning by mirror averaging.
\newblock \emph{Ann. Statist.}, 36\penalty0 (5):\penalty0 2183--2206, 2008.
\newblock ISSN 0090-5364.
\newblock \doi{10.1214/07-AOS546}.
\newblock URL \url{http://dx.doi.org/10.1214/07-AOS546}.

\bibitem[Leung and Barron(2006)]{leung2006information}
Gilbert Leung and Andrew~R. Barron.
\newblock Information theory and mixing least-squares regressions.
\newblock \emph{IEEE Trans. Inform. Theory}, 52\penalty0 (8):\penalty0
  3396--3410, 2006.
\newblock ISSN 0018-9448.
\newblock \doi{10.1109/TIT.2006.878172}.
\newblock URL \url{http://dx.doi.org/10.1109/TIT.2006.878172}.

\bibitem[Mairal and Yu(2012)]{mairal2012complexity}
Julien Mairal and Bin Yu.
\newblock Complexity analysis of the lasso regularization path.
\newblock \emph{arXiv preprint arXiv:1205.0079}, 2012.

\bibitem[Meinshausen and Yu(2009)]{meinshausen2009lasso}
Nicolai Meinshausen and Bin Yu.
\newblock Lasso-type recovery of sparse representations for high-dimensional
  data.
\newblock \emph{Ann. Statist.}, 37\penalty0 (1):\penalty0 246--270, 2009.
\newblock ISSN 0090-5364.
\newblock \doi{10.1214/07-AOS582}.
\newblock URL \url{http://dx.doi.org/10.1214/07-AOS582}.

\bibitem[Pham et~al.(2014)Pham, Ghaoui, and Fernandez]{pham2014robust}
Vu~Pham, Laurent~El Ghaoui, and Arturo Fernandez.
\newblock Robust sketching for multiple square-root lasso problems.
\newblock \emph{arXiv preprint arXiv:1411.0024}, 2014.

\bibitem[Rigollet and Tsybakov(2011)]{rigollet2011exponential}
Philippe Rigollet and Alexandre Tsybakov.
\newblock Exponential screening and optimal rates of sparse estimation.
\newblock \emph{Ann. Statist.}, 39\penalty0 (2):\penalty0 731--771, 2011.
\newblock ISSN 0090-5364.
\newblock \doi{10.1214/10-AOS854}.
\newblock URL \url{http://dx.doi.org/10.1214/10-AOS854}.

\bibitem[Rigollet and Tsybakov(2012)]{rigollet2012sparse}
Philippe Rigollet and Alexandre~B. Tsybakov.
\newblock Sparse estimation by exponential weighting.
\newblock \emph{Statist. Sci.}, 27\penalty0 (4):\penalty0 558--575, 2012.
\newblock ISSN 0883-4237.
\newblock \doi{10.1214/12-STS393}.
\newblock URL \url{http://dx.doi.org/10.1214/12-STS393}.

\bibitem[Tsybakov(2014)]{tsybakov2014aggregation}
A.B. Tsybakov.
\newblock Aggregation and minimax optimality in high-dimensional estimation.
\newblock In \emph{Proceedings of the International Congress of
  Mathematicians}, Seoul, 2014.
\newblock To appear.

\bibitem[van~de Geer and Lederer(2013)]{van2013lasso}
Sara van~de Geer and Johannes Lederer.
\newblock The {L}asso, correlated design, and improved oracle inequalities.
\newblock In \emph{From probability to statistics and back: high-dimensional
  models and processes}, volume~9 of \emph{Inst. Math. Stat. (IMS) Collect.},
  pages 303--316. Inst. Math. Statist., Beachwood, OH, 2013.
\newblock \doi{10.1214/12-IMSCOLL922}.
\newblock URL \url{http://dx.doi.org/10.1214/12-IMSCOLL922}.

\end{thebibliography}
